\documentclass[11pt,leqno,twoside]{amsart}

\usepackage[margin=1.25in]{geometry}

\usepackage[normalem]{ulem}
\usepackage{amsmath,amscd,amssymb,amsfonts,latexsym,wasysym, mathrsfs, mathtools,hhline,color, tikz}
\usepackage[bb=boondox]{mathalpha}

\usepackage[all, cmtip]{xy}
\usepackage{csquotes}
\usepackage{url}
\usepackage{comment}
\usepackage[utf8]{inputenc}

\definecolor{hot}{RGB}{65,105,225}

\usepackage[pagebackref=true,colorlinks=true, linkcolor=hot ,  citecolor=hot, urlcolor=hot]{hyperref}
\usepackage{microtype,tikz-cd,enumitem}
\usepackage{cleveref}

\usepackage{geometry}
\usepackage{chngcntr}

\usepackage{geometry}
\theoremstyle{plain}

\newtheorem{theorem}{Theorem}[section]  
\newtheorem{lemma}[theorem]{Lemma}
 
\newtheorem{proposition}[theorem]{Proposition}
\newtheorem{corollary}[theorem]{Corollary}

\theoremstyle{definition}
\newtheorem{definition}[theorem]{Definition} 
\theoremstyle{remark}
\newtheorem{remark}[theorem]{Remark} 
\newtheorem{example}[theorem]{Example}

\newcommand{\Z}{\mathbb{Z}}
\newcommand{\C}{\mathbb{C}}

\newcommand{\Q}{\mathbb{Q}}
\newcommand{\cR}{\mathcal{R}}
\newcommand{\cV}{\mathcal{V}}
\newcommand{\orb}{\mathrm{orb}}
\newcommand{\ab}{\mathrm{ab}}

\newcommand{\bbmk}{\mathbb{k}}
\newcommand{\Char}{{\rm Char}}

\newcommand{\F}{\mathbb{F}}
\newcommand{\rank}{{\mathrm{rank}}}

\newcommand{\lcm}{{\mathrm{lcm}}}
\newcommand{\Hom}{{\mathrm{Hom}}}

\newcommand{\im}{\mathrm{im}}
\newcommand{\TC}{\mathrm{TC}}

\title[Quasi-projectivity of weighted right-angled Artin groups]{Quasi-projectivity of weighted right-angled Artin groups}
\author{Yongqiang Liu}
\address{Institute of Geometry and Physics, University of Science and Technology of China, Hefei 230026, P.R. China}
\email{liuyq@ustc.edu.cn}
\author{Xiaobin Xu}
\address{School of Mathematical Sciences, University of Science and Technology of China, 96 Jinzhai Road, Hefei, Anhui 230026, P.R. China}
\email{fx920704594@mail.ustc.edu.cn}

\date{\today}


\subjclass[2010]{}

\begin{document}

\begin{abstract}
Weighted right-angled Artin groups generalize right-angled Artin groups by adding weights to the edges of the defining graphs. We classify the weighted right-angled Artin groups that occur as fundamental groups of smooth complex quasi-projective varieties. They are precisely the finite direct products of finitely generated free groups and groups of the form $\langle a,b\mid [a,b]^m=1\rangle$.
\end{abstract}

\maketitle


\section{Introduction}
\subsection{Serre's problem}
A finitely presented group is called \emph{projective} (resp. \emph{quasi-projective}) if it is the fundamental group of a connected smooth complex projective (resp. quasi-projective) variety. The problem of characterizing these groups goes back to Serre and remains widely open. See \cite{ABCKT,Py} for background in the compact K\"ahler setting. A natural approach is to fix a class of finitely presented groups and determine which of its members admit such geometric realizations.

The quasi-projective problem is generally harder to study because the noncompact setting allows phenomena excluded by compact K\"ahler geometry. For a smooth projective variety $X$, Hodge theory gives a pure Hodge structure of weight one on $H^1(X,\Q)$ and forces its dimension to be even. The formality theorem of Deligne, Griffiths, Morgan, and Sullivan supplies further restrictions on the rational homotopy type of $X$ \cite{DGMS}. For a smooth quasi-projective variety $X$, $H^1(X,\Q)$ may have both weight-one and weight-two parts. Such varieties need not be formal. Consequently, the constraints available in the projective setting do not extend unchanged to quasi-projective groups.

Morgan's work on the algebraic topology of smooth algebraic varieties provides a fundamental framework for the noncompact problem \cite{Mor,Mor2}. His mixed Hodge theory for minimal models imposes restrictions on the Malcev Lie algebras of fundamental groups. He also produced the first example of a finitely presented group that is not quasi-projective. Kapovich and Millson subsequently used representation varieties to prove that infinitely many pairwise non-isomorphic Artin groups can not be quasi-projective groups \cite{KM}.

Degree one characteristic varieties give particularly effective obstructions for quasi-projective groups. Beauville studied their structure for projective groups \cite{Bea}. Arapura  extended it to quasi-projective groups \cite{Ara97}. Artal Bartolo, Cogolludo-Agust\'in, and Matei further refined Arapura's results \cite{ACM13}. Dimca, Papadima, and Suciu developed the resulting isotropicity and resonance obstructions \cite{DPS}. These results connect the geometry of orbifold fibrations with invariants computable from a group presentation. Artin type groups are a natural class in which to apply these methods.


\subsection{Weighted right-angled Artin groups}
 Let us first recall the definition of right-angled Artin groups.
\begin{definition}\label{def:raag}
 Let \(\Gamma=(V,E)\) be a finite simple graph with $V$ the vertex set and $E$ the edge set. Its associated \emph{right-angled Artin group} is
\[
G_\Gamma=\langle v\in V\mid [v,w]=1\text{ for every }\{v,w\}\in E\rangle,
\]
where $[v,w]=vwv^{-1}w^{-1}$.
\end{definition}
The class of right-angled Artin groups contains finitely generated free groups and finitely
generated free abelian groups, and is closed under direct products. 
Dimca, Papadima, and Suciu obtained a complete solution of Serre's problem within this
class \cite[Theorem~11.7]{DPS}. Their result will be recalled below, immediately before
the statement of our main theorem.  Related realization problems have also been considered for Artin  
groups. For example, Blasco-Garc\'ia and Cogolludo-Agust\'in classified quasi-projective even
Artin groups \cite{BC}. 

Weighted right-angled Artin groups were introduced in \cite{LL} as a labeled-graph
generalization of right-angled Artin groups.
\begin{definition}\label{def:weighted}
 Let $\Gamma_\ell= (V, E, \ell)$ be a finite simple graph, with vertex set $V$, edge set $E$ and an edge weight function $\ell \colon E \to \Z_{>0}$.  The associated \emph{weighted right-angled Artin group} is
\[
G_{\Gamma_\ell}=\langle v\in V\mid [v,w]^{\ell(\{v,w\})}=1
\text{ for every edge }\{v,w\}\in E\rangle.
\]
An edge is called \emph{heavy} if its weight is greater than one. When all weights are one, $G_{\Gamma_\ell}$ is the right-angled Artin group $ G_\Gamma$ for the underlying graph $\Gamma$.
\end{definition}

\begin{example} \label{ex:join}
    Let $\Gamma_\ell= (V, E, \ell)$ and $\Gamma'_{\ell'}= (V', E', \ell')$ be two edge weighted graphs. Denote by $\Gamma_\ell * \Gamma'_{\ell'}$ their join, with  the vertex set $V\sqcup V'$, the edge set $E \sqcup E' \sqcup \{v,v'| v\in V, v'\in V'\}$ and label $1$ on each edge $\{v,v'\}$. Then $$ G_{\Gamma_\ell * \Gamma'_{\ell'}}= G_{\Gamma_\ell} \times G_{ \Gamma'_{\ell'}}.$$
    If $\Gamma_\ell$ is  discrete  (i.e., $E=\emptyset$), $G_{\Gamma_\ell}$ is a finitely generated free group. More
generally, if $\Gamma_\ell$ is complete multipartite (i.e., a finite join of discrete graphs), then $G_{\Gamma_\ell}$ is a finite direct product of finitely generated free groups.
\end{example}

Liu and Liu classified the K\"ahler weighted right-angled Artin groups and asked for a classification in the quasi-K\"ahler setting \cite[Question~1.8]{LL}. The present work gives a complete answer in the algebraic, quasi-projective category.
\subsection{Main results}
The following is the quasi-projective consequence of \cite[Theorem~11.7]{DPS} due to Dimca, Papadima and Suciu. 
\begin{theorem}\cite{DPS}\label{thm:DPS}
  Let $G_\Gamma$ be the right-angled Artin group associated to a finite simple graph
$\Gamma$.  The following conditions are equivalent:
\begin{enumerate}[label=(\roman*)]
\item $G_\Gamma$ is quasi-projective;
\item $G_\Gamma$ is a finite direct product of finitely generated free groups;
\item $\Gamma$ is complete multipartite, that is, a finite join of discrete graphs.
\end{enumerate}
\end{theorem}

For $r\geq1$, let $D_r$ denote the discrete graph on $r$ vertices. For $m\geq1$, let $S_m$ denote the graph with two vertices and one edge of weight $m$, and put
\[
Q_m=\langle a,b\mid [a,b]^m=1\rangle.
\]
Our main theorem gives the weighted extension of Theorem~\ref{thm:DPS}.
\begin{theorem}\label{thm:main}
 Let $G_{\Gamma_\ell}$ be the weighted right-angled Artin group associated to an
edge weighted finite graph $\Gamma_\ell$.  The following conditions are equivalent:
\begin{enumerate}[label=(\roman*)]
\item $G_{\Gamma_\ell}$ is quasi-projective;
\item $G_{\Gamma_\ell}$ is a finite product of groups of type $\F_r$ with $r\geq 1$ or $Q_m$ with $m\geq 1$;
\item $\Gamma_\ell$ is a finite join of graphs of the form $D_r$ or $S_m$.
\end{enumerate}
\end{theorem}

The classification also identifies the arrangement groups in this class.
\begin{corollary}[see Corollary~\ref{cor arrangement}]\label{cor:intro-arrangement}
For a weighted right-angled Artin group $G_{\Gamma_\ell}$, the following conditions are equivalent:
\begin{enumerate}[label=(\roman*)]
\item $G_{\Gamma_\ell}$ is the fundamental group of a complex hyperplane arrangement complement;
\item $G_{\Gamma_\ell}$ is a finite direct product of finitely generated free groups;
\item $\Gamma$ is complete multipartite and every edge has weight one.
\end{enumerate}
\end{corollary}
Thus, among weighted right-angled Artin groups, the arrangement groups are exactly the
ordinary right-angled Artin groups occurring in \Cref{thm:DPS}. 
 By contrast, the projective groups in the weighted class are precisely the finite products of the groups $Q_m$, see Corollary~\ref{cor projective}. This agrees with the K\"ahler classification in \cite{LL}. 

\subsection{Outline of the proof  and comparison with earlier work}
The realization direction is straightforward. 
 Both $\F_r$ and $Q_m$ are quasi-projective, and quasi-projective groups are closed under
finite direct products.  The necessity of the graph condition is proved in three steps.
\begin{enumerate}
    \item[Step 1. ]  Following the proof of \Cref{thm:DPS} in \cite{DPS},  we use complex characteristic variety to show that the underlying graph of a quasi-projective weighted right-angled Artin group must  be complete multipartite.
    \item[Step 2. ]  We use positive
characteristic  jump loci to show that a heavy edge can not meet a multipartite block containing at
least two vertices.
    \item[Step 3. ]  Using the technique of double cover, we show that two heavy edges can not share a vertex.
\end{enumerate}


Step 2 is the main new ingredient. The weights of $\Gamma_\ell$ change the realization problem in a way that ordinary complex jump loci do not detect. We instead choose a coefficient field of positive characteristic dividing the
weight of the edge. Then  Fox calculus yields a lower bound on twisted first cohomology  that exceeds the generic dimension supplied by the Hochschild–Serre sequence.   

These three steps yield exactly the decomposition in
\Cref{thm:main}. They also explain how several different tools are used. The
ordinary complex jump loci determine the complete multipartite structure, positive
characteristic jump loci detect the weights, and the double cover Hodge argument controls how the heavy edges may intersect.

This approach differs from the proof of the K\"ahler classification in \cite{LL}, which uses integral
homology jump loci, tropical varieties, and Bieri–Neumann–Strebel–Renz invariants. Indeed, our approach  gives an alternative proof of the K\"ahler classification, see Remark \ref{rem:kahler}. 
\subsection{Organization}
This paper is organized as follows. In \cref{pre}, we recall the required background on
characteristic and resonance varieties, compute these invariants for weighted right-angled
Artin groups, and review the structure of degree one jump loci of smooth
quasi-projective varieties in terms of orbifold fibrations.  Section \ref{pfmain}  is devoted to the 
proof of \Cref{thm:main}. Its first three subsections correspond to the three steps
described above, and the final subsection assembles them and derives the applications to
arrangement groups and projective groups.

\subsection*{Acknowledgments}
Both authors are supported by NSFC grant No.~12571047 and the Project of Stable Support for Youth Team in Basic Research Field, CAS (YSBR-001).

\subsection*{AI disclosure}
 During the preparation of this work, the authors used AI for the purpose of language polishing and checking the correctness of the proofs to improve readability.   All mathematical results, claims, and conclusions presented in this paper have been rigorously proved by the authors without the assistance of artificial intelligence. The authors take full responsibility for the content and integrity of this work.

\section{Preliminaries}\label{pre}

\subsection{Characteristic varieties and resonance varieties} 
Let $X$ be a connected finite CW-complex with $\pi_1(X)=G$. 
Fix an algebraically closed field $\bbmk$. The group of $\bbmk$-valued characters, $\mathrm{Char}(G,\bbmk) \coloneqq\mathrm{Hom}(G, \bbmk^*)$, is a commutative affine algebraic group. Since $\bbmk^*$ is abelian, we have $\mathrm{Hom}(G, \bbmk^*)\cong \mathrm{Hom}(\ab(G),\bbmk^*)$, where $\ab(G)$ denotes the abelianization of $G$. We always assume that the free abelian part of $ \ab(G)$ has positive rank.
 Each character $\rho \in \mathrm{Char}(G, \bbmk)$ defines a rank one local system on $X$, denoted by $L_{\rho}$.  
\begin{definition}
The degree one characteristic variety of $X$ with $\bbmk$-coefficients is the set
    \begin{equation*}
    \cV^1(X,\bbmk) \coloneqq \{\rho \in \mathrm{Char}(G,\bbmk) \mid H^1(X, L_\rho) \neq 0\}.
    \end{equation*}
\end{definition}
Since the characteristic variety relies only on the fundamental group $G$, we also denote it as $\cV^1(G,\bbmk)$. 
One may use Fox calculus to compute $\cV^1(G,\bbmk)$ explicitly as follows. Given a presentation $G=\langle x_1,\dots,x_n | r_1, \dots, r_m \rangle$,
there are linear operators $\dfrac{\partial}{\partial x_j}$ for $1\leq j\leq n$, uniquely defined by the following rules:
\begin{center}
$ \frac{\partial 1}{\partial x_j}=0$,
 $ \frac{\partial x_i}{\partial x_j}= \delta_{ij}$, and $  \frac{\partial g_1 g_2}{\partial x_j}= \frac{\partial g_1}{\partial x_j}+g_1 \frac{\partial g_2}{\partial x_j}$.     
\end{center}
Then 
one associates the \emph{Alexander matrix}:
\begin{equation*}
A(\rho) = \left( \rho ( \frac{\partial r_i}{\partial x_j} )\right)_{1 \leqslant i \leqslant m, \; 1 \leqslant j \leqslant n}.
\end{equation*}

\begin{proposition}\cite[Corollary~2.4.3]{Hir}
With the above notations and assumptions, the characteristic varieties $\cV^1(G,\bbmk)$ can be computed by the Alexander matrix as follows:
    \begin{equation*}
    \cV^1(G,\bbmk) = \{\rho \in \mathrm{Char}(G,\bbmk) \mid \rank(A(\rho)) < n-1\} \cup \{\mathbf{1}\},
    \end{equation*}
    where $\mathbf{1}$ denotes the constant sheaf with $\bbmk$-coefficients of $G$.
\end{proposition}

 Every group homomorphism $\varphi\colon G\to Q$ 
induces a morphism between character groups, 
$\hat \varphi\colon \Char(Q,\bbmk) \to \Char(G,\bbmk)$,  
given by $\hat \varphi (\rho)(g)=\rho(\varphi(g))$ for any $g\in G$.  
We list two folklore results that will be used later. 
\begin{lemma} \label{lem from char to H^1} For a group epimorphism $\varphi\colon G \twoheadrightarrow Q$ between finitely presented groups,  consider the two induced injective maps $\hat \varphi\colon \Char(Q,\C) \to \Char(G,\C)$ and $\varphi^* \colon H^1(Q,\C) \to H^1(G,\C) $. Then we have  
   $$\TC_1 (\im \hat  \varphi) = \im \varphi^* ,$$                                                         where $\TC_1$ denotes the tangent cone at the identity element. 
\end{lemma}

\begin{lemma} \label{lem from char to group}
Assume that $ \varphi_i\colon G\twoheadrightarrow Q_i$ for $i=1,2$ are two group epimorphisms  between finitely generated free abelian groups. 
    Then the  two induced embeddings $\hat \varphi_i\colon \Char(Q_i,\C) \to \Char(G,\C)$ for $i=1,2$ have the same image, if and only if, there exists a group isomorphism $\phi \colon Q_1\to Q_2 $ such that $\varphi_2=\phi \circ \varphi_1$. In that case, the images of the two induced injective maps $\varphi_i^*\colon \Hom(Q_i,\Q) \to \Hom(G,\Q) $ for $i=1,2$ coincide. 
\end{lemma}
\begin{proof} The second claim follows from Lemma \ref{lem from char to H^1}. For the first claim, we only need to prove the only if part. Note that $\im \hat  \varphi_i = \{\rho \in \mathrm{Char}(G,\C) \mid \rho|_{\mathrm{ker \varphi_i}} = 1\}$. It suffices to show that the two induced embeddings $\hat \varphi_i\colon \Char(Q,\C) \to \Char(G,\C)$ for $i=1,2$ have the same image only if $\mathrm{ker}\varphi_1=\mathrm{ker}\varphi_2$. 
Characters of a free
abelian group separate its elements. Thus equality of the character images implies equality
of the kernels, and the required isomorphism follows by passing to the common quotient of $G$.
\end{proof}

The next lemma indicates a functoriality property for the characteristic varieties of groups.
\begin{lemma} \cite[Corollary~A.1]{Suc14B}
\label{lem: funct}
Let $\varphi\colon G \twoheadrightarrow Q$ be an epimorphism 
from a finitely generated group $G$ to a group $Q$. For any $\rho \in \Char(Q,\bbmk) $,
we have an injective map $$ H^1(Q,L_\rho) \hookrightarrow H^1(G,L_{\rho\cdot \varphi}).$$
In particular, the induced monomorphism between character groups, 
$\hat\varphi\colon \Char(Q,\bbmk) \to \Char(G,\bbmk)$, restricts to an embedding 
$\cV^1(Q,\bbmk) \hookrightarrow \cV^1(G,\bbmk)$.
\end{lemma}

\medskip

Next we recall the definition of the resonance variety of a finite CW-complex $X$.  
For any $a \in H^1(X,\C)$, we have $a\cup a=0$. Thus, right-multiplication by $a$ defines a cochain complex
\[
(H^*(X,\C), \cup a) \colon  0 \longrightarrow H^0(X,\C) \overset{\cup a}{\longrightarrow} H^1(X,\C) \overset{\cup a}{\longrightarrow} H^2(X,\C) \overset{\cup a}{\longrightarrow} \cdots,
\]
known as the {\it Aomoto complex}.
\begin{definition}\label{def:res} 
    The degree one resonance variety of a finite CW-complex $X$ with complex coefficients is defined as:
    $$\mathcal{R}^1(X,\C) = \{a \in H^1(X,\C)|   H^1(H^*(X,\C),\cup a) \neq 0 \}.$$
\end{definition}
 Since $\cR^1(X,\C)$ only depends on the fundamental group $G$, we also denote it as $\cR^1(G,\C)$. The resonance varieties of groups also exhibit the functoriality property, see e.g. \cite[Corollary~A.1]{Suc14B}. Furthermore, we list the following result that will be used later.
\begin{lemma} \cite[Corollary~B.4]{Suc14B}\label{lem cover} Let $G'$ be a normal subgroup of $G$ of finite index. 
Then the inclusion map $i\colon G'\to G$ induces a linear map $H^1(G,\C) \to H^1(G',\C)$.
If this linear map is an isomorphism, then it identifies $ \cR^1(G',\C)$ and $\cR^1(G,\C)$.
\end{lemma}

The characteristic variety and resonance variety are related by the tangent cone theorem due to Libgober. 
\begin{theorem}\cite{Lib} \label{tangent cone theorem} 
    For a finitely presented group $G$, we have the following tangent inclusion
   \begin{equation} \label{eq tangent cone}
       \TC_1 \cV^1(G,\C) \subseteq \cR^1(G,\C).
   \end{equation}
\end{theorem}

\subsection{Jump loci of weighted right-angled Artin groups}

The characteristic variety and the resonance variety of  right-angled Artin groups are computed  by Dimca, Papadima and Suciu, see  \cite[Proposition 11.5]{DPS} and \cite[Theorem~5.5]{PS06}.

For a graph with vertex set $V$, the abelianization of either $G_\Gamma$ or $G_{\Gamma_\ell}$ is $\Z^V$. For $W\subseteq V$, set
\[
T_W=\{\rho\in(\C^*)^V\mid\rho(v)=1\text{ for }v\notin W\},\qquad
L_W=\mathrm{span}_\C\{v^*\mid v\in W\},
\]
where $v^*$ is dual to the generator $v$. 

\begin{theorem}\cite[Theorem~5.5]{PS06} \label{thm RAAG}
Let \(\Gamma = (V, E)\) be a finite graph with the corresponding right-angled Artin group denoted by $G_\Gamma$. Then we have 
\begin{center}
       $ \cV^1(G_\Gamma,\C)= \bigcup_{W\subseteq V} T_W$ and $\cR^1(G_{\Gamma}, \C) = \bigcup_{W \subseteq V} L_W,$
\end{center}
    where       both unions are over all subsets $W$ of vertices $V$ such that the induced subgraph $\Gamma(W)$ is maximally disconnected. In particular, the tangent cone inclusion (\ref{eq tangent cone}) holds as equality. 
\end{theorem}
By an observation, one can generalize this theorem to weighted right-angled Artin group as in \cite[Corollary~6.7]{LL}.
\begin{corollary} \label{cor  char}
     Let $G_{\Gamma_\ell}$ be the weighted right-angled Artin group associated to an
edge weighted finite simple graph $\Gamma_\ell$.  Then we have 
\begin{center}
       $ \cV^1(G_{\Gamma_\ell},\C)= \bigcup_{W\subseteq V} \mathbb{T}_W$ and $\cR^1(G_{\Gamma}, \C) = \bigcup_{W \subseteq V} L_W.$
\end{center}
where   
    both unions are over all subsets $W$ of vertices $V$ such that the induced subgraph $\Gamma(W)$ is maximally disconnected. In particular, the tangent cone inclusion (\ref{eq tangent cone}) holds as equality. 
\end{corollary}
\begin{proof}
    For an edge   $\{u,v\}$ with weight $m$,    the abelianization of the Fox derivative of the relation $[u,v]^m=1$  gives the following row in the Alexander matrix: $$\begin{pmatrix}
  m(1-v)  & m(u-1) & 0 & \cdots & 0
\end{pmatrix}.$$
  Let $\Gamma$ denote the underlying graph of $\Gamma_\ell$. 
    Comparing with the row given by the relation $[u,v]=1$ in $G_\Gamma$, they only differs by a multiplication with the nonzero integer $m$. 
    Since $m$ is a unit in $\C$, the two groups $G_{\Gamma_\ell}$ and $G_\Gamma$ get the same Fox calculus over $\C$. Hence   $\cV^1(G_{\Gamma_\ell},\C)$  and $\cV^1(G_\Gamma,\C)$ agree  and the claim follows from \Cref{thm RAAG}.

    Since $G_{\Gamma_\ell}$ admits a commutator-relators presentation, the relation $[u,v]^m=1$ gives the following row of the linearized Alexander matrix
    $$\begin{pmatrix}
  -m z_v  & m z_u & 0 & \cdots & 0
\end{pmatrix}$$
with entries in the polynomial ring $\C[z_w|w\in V]$. 
Since the resonance variety of a group can be computed by the linearized Alexander matrix (see e.g. \cite[Section 3.8]{MS}), then  $\cR^1(G_{\Gamma_\ell},\C)$ and $\cR^1(G_\Gamma,\C)$ agree for the same reason and the claim  follows from \Cref{thm RAAG}.
\end{proof}

\begin{remark} \label{rem cup}
        Suppose $u,v$ are two vertices of $\Gamma_\ell$ (also viewed as generators of $G_{\Gamma_\ell}$) and $u^*,v^*$ represent their dual in $H^1(G_{\Gamma_\ell},\Q)$. Then $\{u,v\} \in E$ if and only if $u^*\cup v^*  \neq 0$.
In fact, suppose $\{u,v\} \in E$ and $\ell(\{u,v\}) = m \geqslant 1$. By \cite[Theorem~2.4]{FS}, $(u^* \cup v^*,e) = m$ in $H^*(X,\Q)$. Here $X$ is the $2$ dimensional cell associated to the representation by vertices and edges, and $e$ represents the $2$-cell given by relation $[u,v]^m$. 
Consider the natural inclusion $\iota: X \to K(G_{\Gamma_\ell},1)$, which induces an injective map
    $$ \iota^*\colon H^2(G_{\Gamma_\ell},\Q) \to H^2(X,\Q).$$
    Functoriality of cup product implies that $u^*\cup v^* \neq 0$ in $H^2(G_{\Gamma_\ell},\Q)$. 
    On the other hand, suppose $\{u,v\} \notin E$. Then $u^*\cup v^* = 0$ in $H^2(X,\Q)$.  Since $\iota^*\colon H^2(G_{\Gamma_\ell},\Q) \to H^2(X,\Q)$ is injective, $u^*\cup v^* = 0$ in $H^2(G_{\Gamma_\ell},\Q)$.

    Moreover, the cup-product classes corresponding to distinct edges are linearly independent:
their evaluations are supported on distinct edge cells. In particular, cup products supported on disjoint sets of edges have disjoint
images in $H^2(G_{\Gamma_\ell},\Q) .$
\end{remark}

\subsection{Quasi-projective groups}
\begin{definition}
    A group $G$ is quasi-projective if $G=\pi_1(X)$ for a connected complex smooth quasi-projective variety $X$.
\end{definition}
We list the following well-known properties for quasi-projective groups. 
\begin{proposition}\label{prop property}
\begin{itemize}
\item[(1)] If $G$ is a quasi-projective group and $K \subset G$ is a finite-index subgroup of $G$, then $K$ is also a quasi-projective group and the corresponding finite unramified cover is an algebraic map.
\item[(2)] Finite direct products of quasi-projective groups are quasi-projective.
\end{itemize}
The same assertions hold for projective groups.
\end{proposition}

A celebrity result  due to Beauville \cite{Bea}, Arapura \cite{Ara97} and Artal Bartolo, Cogolludo-Agust\'in and Matei \cite{ACM13} puts  strong constraints for its fundamental group. To explain their results, we first recall the definition of orbifold fibrations. 

Let $(C,\bar m)$ be a complex smooth curve of genus 
$g\ge 0$, with $s\ge 0$ points removed,  and with $n$ marked points 
$\{q_1,\dots, q_n\} $ with the multiplicity $m_j$ for $q_j$. Here  $\bar m=( m_1,\dots, m_n)$ and  $m_j>1$ for all $1\leq j \leq n$. Let $X$ be a complex smooth quasi-projective variety.  A surjective, algebriac map 
$f\colon X\to C$ is called an {\it orbifold fibration} if the generic fiber is connected, $2g+s>1$ and $f$ has  exactly multiple fibers over the $n$ marked points $\{q_1,\cdots, q_n\}$ with  the multiplicity $m_j$  for the fiber   $f^{*}(q_j)$  (the $\gcd$ of the coefficients of the divisor $f^* q_j$).

\begin{definition}
The \emph{orbifold fundamental group} of \((C, \bar m)\) is defined as
\[
\pi_1^{\mathrm{orb}}(C,\bar m)
\coloneqq
\pi_1\bigl(C\setminus\{q_1,\ldots,q_n\}\bigr)
/\left\langle \mu_j^{m_j}=1, j=1,\cdots,n\right\rangle,
\]
where \(\mu_j\) is a meridian of \(q_j\).  If \(C\) is a compact surface of genus \(g\), 
then this group has presentation
\[
\begin{aligned}
\Bigg\langle
a_1,\ldots,a_g,b_1,\ldots,b_g,\mu_1,\ldots,\mu_n
\ \Bigg|\
\prod_{i=1}^{g}[a_i,b_i]\cdot\prod_{j=1}^{n}\mu_j=1,\ 
\mu_j^{m_j}=1,\ j=1,\ldots,n
\Bigg\rangle .
\end{aligned}
\]
If \(C\) is not compact, then  \(\pi_1(C)\) is a free group of  \(r= 2g+s-1\) generators and
\[
\pi_1^{\mathrm{orb}}(C,\bar m)
=
\left\langle
a_1,\ldots,a_r,\mu_1,\ldots,\mu_n
\ \middle|\
\mu_j^{m_j}=1,\ j=1,\ldots,n
\right\rangle .
\]
\end{definition}

In both cases, the characteristic varieties can be computed directly via fox calculus. 
\begin{proposition}\label{prop:orb}\cite[Propostion 2.10, 2.11]{ACM13} \label{prop orbifold group}
Let $G=\pi_1^\orb(C,\bar m)$ and let $\Char^0(G,\C)$ denote the identity component of the character group $\Char(G,\C).$ 
\begin{itemize}
\item[(1)]  Assume that $C$ is compact of genus $g\geq 1$.  Then the torsion part of $\ab(\pi_1^{\mathrm{orb}}(C, \bar m))$ 
has order $(\prod_{j=1}^n m_j) /\lcm(m_1,\cdots,m_n)$ and 
\[
\cV^1(G, \C)=\begin{cases}
\mathrm{Char}(G,\C),
& \text{if } g \geqslant 2, \\[0.4em]
\{\mathbf{1}\}\cup (\Char(G,\C)\setminus \Char^0(G,\C)),
& \text{if } g=1 \text{ and}  \prod_{j=1}^n m_j >\lcm(m_1,\cdots,m_n), \\[0.8em]
\{\mathbf{1}\},
& \text{if } g=1 \text{ and}  \prod_{j=1}^n m_j =\lcm(m_1,\cdots,m_n). \\[0.8em]
\end{cases}
\]

\item[(2)] Assume that $C$ is non-compact.  Then the torsion part of $\ab(\pi_1^{\mathrm{orb}}(C, \bar m))$ 
has order $\prod_{j=1}^n m_j$ and 
\[
\cV^1(G,\C)=\begin{cases}
\mathrm{Char}(G,\C),
& \text{if } r \geqslant 2, \\[0.4em]
\{\mathbf{1}\}\cup (\Char(G,\C)\setminus \Char^0(G,\C)),
& \text{if } r=1. \\[0.8em]
\end{cases}
\]
In particular, if $r\geq 2$, $\cV^1(\pi_1^{\mathrm{orb}}(C,\bar m),\C) $ has $\prod_{j=1}^n m_j$ many parallel components; while if $r=1$, $\cV^1(\pi_1^{\mathrm{orb}}(C,\bar m),\C) $ has $(\prod_{j=1}^n m_j) -1$ many parallel components.
\end{itemize}
\end{proposition}

An orbifold fibration induces an 
epimorphism on fundamental groups (see e.g. \cite[Lemma 3]{CKO}): $$f_* \colon \pi_1(X) \twoheadrightarrow \pi_1^{\orb}(C, \bar m).$$
By Lemma \ref{lem: funct}, 
it induces an embedding $$\cV^1(\pi_1^{\orb}(C, \bar m),\bbmk)\hookrightarrow \cV^1(X,\bbmk).$$ 
The following structure theorem shows that every positive dimensional component of $\cV^1(X,\C)$ can be realized by the orbifold fibrations in this way. 

\begin{theorem}[\cite{ACM13}]\label{thm:structure}
Let $X$ be a complex smooth quasi-projective variety. Then we have
\[
\cV^1(X,\C)=\bigcup_{f }{f^*} \cV^1(\pi_1^{\orb}(C,{\bar m}),\C)\cup Z,
\]
where $Z$ is a finite set of torsion points and the union runs over all surjective orbifold fibrations $f\colon  X\rightarrow C$. 
Moreover, it is a finite union.
\end{theorem}

The following result will play a crucial role in the proof of \Cref{thm:main}.
\begin{theorem} \cite[Lemma 3]{CKO}\label{thm finitely generated}
Consider an orbifold fibration $f\colon X\to C$.  Let $F$ denote the generic fiber of $f$. Then we have an   exact sequence of groups $$ \pi_1(F) \to \pi_1(X) \overset{f_*}{\to} \pi_1^{\orb}(C,{\bar m}) \to 1,  $$
where the first map is induced by the inclusion from $F$ to $X$.
In particular the kernel of $f_*$ is finitely generated.
\end{theorem}

Next we recall a result due to Dimca, Papadima and Suciu, which is closed related to \Cref{thm:structure}.

\begin{definition}
Let $\mu:\bigwedge^2 H^1 \longrightarrow H^2$
be a \(\mathbb{C}\)-linear mapping between complex linear spaces, where \(\dim H^1<\infty\), and let
\(V\subset H^1\) be a \(\mathbb{C}\)-linear subspace.
\begin{enumerate}
    \item[(1)] \(V\) is called \emph{0-isotropic} with respect to \(\mu\) if the restriction
    $    \mu^V:\bigwedge^2 V \longrightarrow H^2$
    is trivial.
    \item[(2)] \(V\) is called \emph{1-isotropic} with respect to \(\mu\) if the restriction $
    \mu^V:\bigwedge^2 V \longrightarrow H^2$    has \(1\)-dimensional image and is a nondegenerate skew-symmetric bilinear form.
\end{enumerate}
\end{definition}

\begin{theorem}\cite[Theorem C]{DPS}\label{thm:DPS-isotropic}
Let \(X\) be a connected smooth quasi-projective variety and let
\(G=\pi_1(X)\). Let $V$ be a positive-dimensional irreducible component of
\(\mathcal V^1(G,\mathbb C)\) containing the identity, and set $
L=\TC_1V\subset H^1(G,\mathbb C).$
Then $L$ is a complex linear vector space and \(L\) is \(p\)-isotropic for some
\(p\in\{0,1\}\).
\end{theorem}
\begin{remark} \label{rem isotropic}
As shown in \cite[Corollary~7.3]{DPS}, if $H^1(X,\Q)$ has pure Hodge structure of weight two, $C$ has to be $\mathbb{CP}^1$ taking out a few  points and $L$ is always $0$-isotropic; while if $X$ is projective, then $C$ is  compact and $L$ is always $1$-isotropic. 
\end{remark}

\begin{lemma}\cite[Lemma~9.4]{DPS}\label{lem:tangent-obstruction}
Let $L=\TC_1 V$ be as in Theorem \ref{thm:DPS-isotropic}. 
Suppose that $L=L'\oplus L''$ with \(L'\neq0\) and \(L''\neq0\), and suppose that the cup product satisfies $ L'\cup L''=0.$  Assume moreover that the two images
\[
\operatorname{im}\bigl(\wedge^2 L'\to H^2(G,\mathbb C)\bigr),
\qquad
\operatorname{im}\bigl(\wedge^2 L''\to H^2(G,\mathbb C)\bigr)
\]
intersect trivially. Then the cup products on both
$\wedge^2 L'$ 
and $\wedge^2 L''$ vanish.
\end{lemma}

\section{Proof of \Cref{thm:main}}\label{pfmain}
\subsection{Step 1}
The following result generalizes its unweighted counterpart in \cite{DPS}. Although the argument is essentially the same, we include a complete proof for the sake of completeness.
\begin{lemma}\label{lem:multipartite}
   If $G_{\Gamma_\ell}$ is  quasi-projective, then the underlying graph $\Gamma$ is  complete multipartite.
\end{lemma}
\begin{proof}
In the proof of this lemma, all the cup products are taking over complex coefficients.

If $\Gamma$ is complete, then it is complete multipartite, with each part consisting of
one vertex. Hence we may assume that $\Gamma$ is not complete.

Let $W\subseteq V$ be maximal with respect to the property that the induced subgraph
$\Gamma(W)$ is disconnected. By Corollary \ref{cor  char}, the coordinate subtorus $T_W$
 is an irreducible component of $\mathcal V_1(G_{\Gamma^\ell},\mathbb C)$ containing the identity with tangent space $L_W$.

Choose one connected component  $W'$ of $\Gamma(W)$, and set $W''=W\setminus W'.$
Both are nonempty. Since there are no edges between $W'$ and $W''$, we have
\[
L_W=L_{W'}\oplus L_{W''}, L_{W'}\cup L_{W''}=0. 
\] 
On the other hand, since the edge sets of \(\Gamma(W')\) and \(\Gamma(W'')\) are disjoint,
the two images 
\[
\wedge^2L_{W'}\to H^2(G_{\Gamma^\ell},\mathbb C) 
 \text{ and } 
\wedge^2L_{W''}\to H^2(G_{\Gamma^\ell},\mathbb C)
\]
must intersect trivially.

Assume that \(G_{\Gamma^\ell}\) is quasi-projective. Applying Lemma \ref{lem:tangent-obstruction} to \(T_W\),
we obtain that
\[
\wedge^2L_{W'}\to H^2(G_{\Gamma^\ell},\mathbb C)
\quad\text{and}\quad
\wedge^2L_{W''}\to H^2(G_{\Gamma^\ell},\mathbb C)
\]
are both zero. Therefore the cup product on \(L_W\) is also zero.
By Remark \ref{rem cup}, 
 the vanishing of
the cup product on \(L_W\) implies that \(\Gamma(W)\) is discrete. Thus every maximal
disconnected induced subgraph of \(\Gamma\) is discrete.

It remains to prove that \(\Gamma\) is complete multipartite. Suppose not. Then there exists a maximal discrete induced subgraph $\Gamma'$ and another vertex that is not adjacent to some vertex of $\Gamma'$. Then $\Gamma'$, together with the extra vertex, is again disconnected, which leads to a contradiction with the assumption that $\Gamma'$ is maximal disconnected.
\end{proof}
\subsection{Step 2}
Complex jump loci determine the
multipartite structure but do not detect the weights. We now use 
local systems with positive characteristic field coefficients to show that heavy edges can not be adjacent to a discrete block of size greater than one.
\begin{lemma}\label{lem:no_cross_weight}
    Let $\Gamma(W)$ be a maximal disconnected subgraph of $\Gamma_\ell$ with at least two vertices. If there is an edge $\{u,v\}$ with weight $m  > 1$, where $ u \in W $ and $v \in V \setminus W$, then $G_{\Gamma_\ell}$ is not a quasi-projective group.
\end{lemma}
\begin{proof}
Assume that $ G_{\Gamma^\ell}\cong \pi_1(X)$
for some complex smooth quasi-projective variety $X$.  By Lemma \ref{lem:multipartite},
the underlying graph $\Gamma$ is complete multipartite. Hence $\Gamma(W)$ is discrete with $r = |W|$ vertices.

 Let $T_W$ be the sub-torus corresponding to \(W\) as in Corollary \ref{cor  char}. 
By \Cref{thm:structure} there exists an orbifold fibration 
$f:X\longrightarrow C$
onto a smooth algebraic curve $C$, with connected generic fiber, such that
\[
T_W=f^*(T).
\]
for  $T$ the identity component  of $\cV^1(\pi_1^\text{orb}(C,\bar m),\C)$.  Furthermore, by \Cref{thm finitely generated} 
\begin{equation} \label{ses 2}
    1\longrightarrow K\longrightarrow G_{\Gamma_\ell}\overset{f_*}{\longrightarrow} \pi_1^\text{orb}(C,\bar m)\longrightarrow 1,
\end{equation}
is exact with $K$ finitely generated. Indeed, $K$ is the image of $\pi_1(F)$ with $F$ being the generic fiber of $f$. 
 Since $\Gamma(W)$ is discrete, the fact that $r-1 \leq \mathrm{dim}_\C H^1(G_{\Gamma_\ell},\C_{f^*\rho}) = \mathrm{dim}_\C H^1(C,\C_\rho)$ for generic character $\rho$ \cite[Thm~6.2]{DPS} forces $C$ to be non-compact.

 We claim that $f$ has no multiple fibers, i. e. $\pi_1^\text{orb}(C,\bar m) = \pi_1(C)$.  Otherwise, the existence of multiple fibers leads to translated sub-torus in $\cV^1(X,\C)$ by \Cref{thm:structure} and Proposition \ref{prop orbifold group}(2). But Corollary \ref{cor  char} shows that $\cV^1(G_{\Gamma_\ell},\C)$ has no translated positive dimensional sub-torus for weighted right-angled Artin group.

Now we have two group epimorphism:  $f_*\colon G_{\Gamma_\ell}\twoheadrightarrow \pi_1(C)\cong \F_r$ in (\ref{ses 2});
 the other one is the natural projection $g\colon G_{\Gamma_\ell} \to G_W\cong \F_r$. Here $G_W$ is the weighted right-angled Artin group corresponding to the sub-graph $\Gamma(W)$ and $G_W=\F_r$ since $W$ is discrete. The natural projection $g$ sends the vertices not in $W$ to the identity element. Since $g$ is  surjective, by Lemma \ref{lem: funct} it induces an embedding $\hat g\colon \cV^1(G_W,\C) \hookrightarrow \cV^1(G_{\Gamma_\ell},\C),$ and $\im \hat g=T_W=f^*T$. Note that the character of a group always factor through the abelianization of the group. Then by Lemma \ref{lem from char to group}, there exists a group isomorphism 
 $\phi\colon H_1(C,\Z) \to H_1(G_W,\Z) $ such that the following diagram commutes:
 \begin{equation} \label{commutative diagram}
     \begin{tikzcd}
     & \pi_1(C) \ar[r,"\ab"] &  H_1(C,\Z) \ar[dd,"\phi"]\\
   G_{\Gamma_\ell} \ar[rd,"g"] \ar[ru,"f_*"] \ar[r,"\ab"] & H_1(G_{\Gamma_\ell},\Z) \ar[rd] \ar[ru] & \\
     & G_W \ar[r,"\ab"]  &  H_1(G_W,\Z),
 \end{tikzcd} 
 \end{equation}
where $\ab$ is the abelianzation map.

Fix a prime $p\mid m$, and let $\bbmk$ be an algebraically closed field of
characteristic $p$. Choose a character
$\rho_\lambda\colon H_1(G_W,\Z)\to \bbmk^*$
such that
\[
\rho_\lambda(u)=\lambda \in \bbmk^* \text{ with } \lambda \neq 1\qquad
\text{ and }\rho_\lambda(x)=1\quad\text{for all } u\neq x \in W,
\]
where $u$ is the chosen vertex in $W$. 

Let $Q$ be the weighted right-angled Artin group corresponding to the induced subgraph  on $W\cup \{v\}$. 
Consider the natural projections $g'\colon G_{\Gamma_\ell} \twoheadrightarrow Q$ by sending the vertex not in this subgraph to $1$. It is obvious that 
$g$ factors through $g'$. Pulling back the local system $\rho_\lambda$ to $G_{\Gamma_\ell}$ and $Q$ and still denote the local system by $L_{\rho_\lambda}$ (it will be clear in the context which local system we are using). Since $g'$ is surjective, by Lemma \ref{lem: funct}
we have $$ \dim H^1(G_{\Gamma_\ell}, L_{\rho_\lambda})\geq \dim H^1(Q, L_{\rho_\lambda}).$$

We now compute \(H^1(Q,L_{\rho_\lambda})\) by Fox calculus. The group \(Q\) has \(r+1\)
generators, namely the vertices of \(W\cup\{v\}\). Since \(W\) is discrete, there are
no relations among vertices of \(W\). The only relations are those corresponding to
edges \(\{w,v\}\), with \(w\in W\). For any \(w\in W\setminus\{u\}\), say the edge $\{w,v\} $ has weight $n$. Then the corresponding relation gives a row $$\begin{pmatrix}
  n(1-v) & 0 & \cdots & 0 & n(w-1) & 0 & \cdots & 0  
\end{pmatrix}$$ in the Alexander matrix. Since $\rho_\lambda(w)=\rho_\lambda(v)=1$, this row vanishes for $\rho_\lambda$. Meanwhile,  the edge $\{u,v\} $ for the group $Q$ gives the row of the Alexander matrix 
$$\begin{pmatrix}
  m(1-v)  & m(u-1) & 0 & \cdots & 0
\end{pmatrix}.$$
Then this  row also vanishes, since  $\operatorname{char}\bbmk=p$ divides $m$.  Thus the evaluated Alexander matrix of \(Q\) at \(\rho_\lambda\) is the zero matrix. Since
\(Q\) has \(r+1\) generators and \(\lambda \ne1\), we get
\[
\dim_\bbmk H^1(Q,L_{\rho_\lambda})
=
(r+1)-1
=
r.
\]
Consequently,  we have
\begin{equation} \label{inequality}
     \dim H^1(G_{\Gamma_\ell}, L_{\rho_\lambda}) \geq r \text{ if } \lambda \neq 1
\end{equation}

On the other hand, let  \(\rho'_\lambda \coloneqq \rho_\lambda  \circ \phi  \colon H_1(C,\Z)\to \bbmk^*\) from the commutative diagram (\ref{commutative diagram}). Then the local system \(f^* L_{\rho'_\lambda}\)
is trivial on \(K\).  The Hochschild--Serre spectral sequence for (\ref{ses 2}) (with $\pi_1(C)=\F_r$)
gives an  exact sequence:
\begin{equation}
    \label{ses}
0\to H^1(\F_r,L_{\rho'_\lambda})\to H^1(G_{\Gamma_\ell},f^*L_{\rho_\lambda})
\to H^0(\F_r,H^1(K,\bbmk)\otimes L_{\rho'_\lambda}) .
\end{equation}
Here $\F_r$ acts on $H^1(K,\bbmk)$ by the dual of conjugation. 
Since \(K\) is finitely generated, \(H^1(K,\bbmk)\) is finite-dimensional. Set $\F_r=\langle x_1,\cdots, x_r\rangle$. Let $A_i$ denote the linear action of $x_i$ on \(H^1(K,\bbmk)\) and assume that $\rho'_\lambda$ sends $x_i$ to $\lambda'_i\in \bbmk^*$. 
Then we have an isomorphism $$H^0(F_r,H^1(K,\bbmk)\otimes L_{\rho'_\lambda}) \cong \bigcap_{i=1}^r \ker(\lambda'_i A_i-\mathrm{id}).$$
Hence $H^0(\F_r,H^1(K,\bbmk)\otimes L_{\rho'_\lambda})=0$ for all but finitely many $(\lambda'_1,\cdots,\lambda'_r)$. 
Since \(\bbmk\) is infinite, via the isomorphism $\phi$ in (\ref{commutative diagram}) we may choose \(\lambda\in \bbmk^*\setminus\{1\}\) such that
$
H^0(\F_r,H^1(K,\bbmk)\otimes L_{\rho'_\lambda})=0.$
Since \(\lambda\ne1\), by the exact sequence  (\ref{ses}) we have
\[
H^1(G_{\Gamma_\ell},L_{\rho_\lambda})=H^1(G_{\Gamma_\ell}, f^*L_{\rho'_\lambda})\cong H^1(\F_r,L_{\rho'_\lambda})=r-1.
\]
This contradicts the inequality (\ref{inequality}). Therefore \(G_{\Gamma^\ell}\) cannot be quasi-projective.
\end{proof}
\subsection{Step 3}\label{index2}
It remains to show that the heavy
edges are pairwise disjoint. Throughout this subsection, let \(G=G_{\Gamma^\ell}\) be a quasi-projective weighted
right-angled Artin group, and suppose that \(\{u,v\}\) is an edge of weight \(m>1\).
Write the remaining vertices as  $\{ x_1,\ldots,x_n\}.$
Let $\alpha\colon G\longrightarrow \mathbb Z_2$
be the group homomorphism defined by
\[
\alpha(u)=\alpha(v)=1\in \mathbb Z_2,\qquad
\alpha(x_i)=0\in \mathbb Z_2.
\]
Set $ G'=\ker \alpha.$ Suppose \(G=\pi_1(X)\), where \(X\) is a smooth quasi-projective variety, and let
$\pi\colon X'\longrightarrow X$
be the corresponding double cover map, with $X'$ also smooth quasi-projective and $\pi_1(X')\cong G'$, see Proposition \ref{prop property}.  Then the induced map
\[
\pi^* \colon H^1(X,\mathbb Q)\longrightarrow H^1(X',\mathbb Q)
\]
is injective and is a morphism of mixed Hodge structures. 

\begin{lemma} \label{lem iso} With the above notations and assumptions, 
    $\pi^*\colon H^1(X,\Q) \to H^1(X',\Q)$ is an isomorphism of mixed Hodge structure.
\end{lemma}
\begin{proof}
We only need to show that $\pi^*$ is an isomorphism between rational vector spaces.
Note that 
      $\pi_*\underline{\Q}_{X'} = \underline{\Q}_X \oplus L_\rho$, where $L_\rho$ is a rank one $\Q$-local system corresponding to the representation $\rho$ of $G$ by sending $u,v$ to $-1$ and all other $x_i$ for $1\leq i\leq n$ to $1$. Since $G$ is quasi-projective, Lemma \ref{lem:multipartite} and Lemma \ref{lem:no_cross_weight} imply that all $x_i$'s are adjacent to both $u$ and $v$. Therefore  $\rho \notin \cV^1(G,\Q)$ by Corollary \ref{cor  char}. Hence $H^1(X,L_\rho) = 0$ and  it implies 
    $$H^1(X',\Q) \cong H^1(X, \pi_*\underline{\Q}_{X'}) \cong H^1(X,\Q) \oplus H^1(X,L_\rho) = H^1(X,\Q).$$
    Then the claim follows. 
\end{proof}

\begin{lemma} \label{lem Hodge} Let \(G=G_{\Gamma^\ell}\) be a quasi-projective weighted
right-angled Artin group, and suppose that \(\{u,v\}\) is an edge of $\Gamma_\ell$ with weight \(m>1\).
Then for  every smooth quasi-projective realization  $G=\pi_1(X)$,
 $$L_{u,v} = \operatorname{span}_{\Q}\{u^*, v^*\} \subseteq H^1(X,\Q)$$ has a pure Hodge structure of weight one.    
\end{lemma}

\begin{proof} We divide the proof to 3 steps. 

\medskip

\noindent Step A.  We start with a toy example.  Let $Q =Q_m= \langle u,v\mid [u,v]^m=1\rangle$ be the weighted right-angled Artin group associated with the graph $S_m$. Let $Q'$ be the kernel of the group homomorphism $Q \to \Z/2\Z$  sending both $u$ and $v$ to $1$. Then $Q'$ has the following presentation, via Reidemeister–Schreier rewriting process \cite[Theorem~2.1]{Fox}:
$$ Q' = \langle \widetilde{u}, a, b | (b \widetilde u^{-1} a ^ {-1}) ^m, (\widetilde u a b ^{-1} ) ^m \rangle, $$ with $\widetilde u=u^2,  a=vu^{-1}, b=uv.$

Note that both $Q$ and $Q'$ are isomorphic to some compact orbifold groups. 
In fact, $Q$ is isomorphic to the compact orbifold group with genus 1 and one marked point of multiplicity $m$: $$\pi_1^\orb(C,m)=\langle x,y, z| z^m, xyx^{-1}y^{-1}z \rangle .$$ 
Meanwhile, $Q'$ is isomorphic to the compact orbifold group with genus 1 and two marked points both of multiplicity $m$: $$\pi_1^\orb(C,(m,m))=\langle x, y, z, w| z^m, w^m, xyx^{-1}y^{-1}zw \rangle .$$
A group isomorphism  $\phi \colon \pi_1^\orb(C,(m,m)) \to Q' $ is given by 
\begin{equation*}
        \phi(x)=a, \phi(y)=\widetilde{u}, \phi(z)=\widetilde{u}ab^{-1},\phi(w)=b \widetilde u^{-1} a ^ {-1}.
\end{equation*}

It is easy to see that $\ab(Q')\cong \Z^2\oplus \Z_m$. Hence $\Char(Q',\C)$ consists of $m$ many connected parallel torus. Write its connected character components as $T_\lambda$, indexed by $\lambda^m=1$, with $T_1$ the identity
component.   By Proposition \ref{prop orbifold group}(1), we have
\begin{equation} \label{toy}
    \cV^1(Q',\C)= \{\mathbf{1}\} \cup \bigsqcup_{\lambda^m=1, \lambda\neq 1} T_\lambda .
\end{equation}
Each $T_\lambda$ has dimension two.

As in Lemma \ref{lem iso}, let $X $ and $X'$ be the complex smooth quasi-projective varieties such that $\pi_1(X)=Q $ and $\pi_1(X')=Q' $. 
Then comparing \Cref{thm:structure} with the computations of $\cV^1(Q',\C)$ in (\ref{toy}), there exists a surjective orbifold fibration $f\colon X' \to C$ with $b_1(C)=2$.  The base $C$ must be compact: otherwise Proposition~\ref{prop orbifold group}(2) would then pull back its entire identity character component into $\cV^1(Q',\C)$, contradicting the isolated identity in (\ref{toy}). Hence $C$ is an elliptic curve. Then the induced map $f^* \colon H^1(C,\Q) \to H^1(X',\Q)$ is injective. Since $\ab(Q') $ has rank 2, $f^*$ induces an isomorphism of mixed Hodge structures. Together with the isomorphism induced by the covering map in Lemma \ref{lem iso}, 
we get the following isomorphisms of mixed Hodge structures $$  L_{u,v}  = H^1(X,\Q) \cong H^1(X',\Q)\cong H^1(C,\Q).$$
Then the claim for this special case follows. 

\medskip

\noindent Step B.  Consider the double cover $\pi\colon X' \to X$ with $\pi_1(X)\cong G$ and $\pi_1(X')\cong G'$ as in Lemma \ref{lem iso}.
Let $\pi_* \colon G'\to G$ denote the induced inclusion map.  Consider the natural projection $p\colon G \to Q=\langle u,v \mid [u,v]^m\rangle$, which sends all the generators $\{x_1, \cdots, x_n\}$ to the identity element and preserves $u $ and $v$. It is clear that the map $ \alpha$ (used to construct $G'$) factors through $p$.   
We have the following commutative diagram:
\[
\begin{tikzcd}
G' \ar[r, "\pi_*"] \ar[d, "p'"] & G \ar[d, "p"] \\
Q' \ar[r, "\pi'_*"] & Q
\end{tikzcd}
\]
It induces the following commutative diagram for $H^1(-,\Q)$:
\[
\begin{tikzcd}
H^1(G',\Q) & H^1(G,\Q)  \ar[l, "\pi^*"]   \\
H^1(Q',\Q) \ar[u, "(p')^*"] & H^1(Q,\Q) \ar[u, "p^*"]\ar[l, "(\pi')^*"]
\end{tikzcd}
\]
Note that the two horizontal maps are both isomorphisms by Lemma \ref{lem iso} and the two vertical maps are both injective. 
In particular, $L_{u,v}= \im p^*  $ is isomorphic to $\im (p')^* $ via $\pi^*$.
Since $\pi^*$ is an isomorphism of mixed Hodge structures, the proof is reduced to show that $\im (p')^*$ has  pure Hodge structure of weight one.

\medskip

\noindent Step C.  By Lemma \ref{lem: funct}, the surjection $p'\colon G'\to Q'$ induces an embedding $$\widehat{p'} \colon \cV^1(Q',\C) \hookrightarrow  \cV^1(G',\C).$$
By (\ref{toy}), for a fixed $\lambda$ with $\lambda^m=1$ and $\lambda\neq 1$, we have $\widehat{p'} (T_\lambda) \subseteq \cV^1(G',\C)$. We claim that  $\widehat{p'}(T_\lambda)$ is indeed an irreducible component of $\cV^1(G',\C)$.
Otherwise, there exists an irreducible component $V$ of  $\cV^1(G',\C)$ such that $ \widehat{p'}(T_\lambda) \subsetneq V$. Since $T_\lambda$ has dimension 2, $\dim_\C V\geq 3$. By  Proposition~\ref{prop:orb} and \Cref{thm:structure},
 either $V$ contains the identity element or $V$ has a parallel component $V'$, which is also an irreducible component of $\cV^1(G',\C)$, containing the identity element. This is precisely where $\dim V\geq 3$ is used. Assume the latter case holds and the proof proceeds with no difference in the former case. By the tangent cone formula in Theorem \ref{tangent cone theorem}, we have
 $$ \TC_1 V' \subseteq \TC_1 \cV^1(G',\C) \subseteq \cR^1(G',\C).$$
 On the other hand, by Lemma \ref{lem iso}, the induced map $\pi^*\colon H^1(G,\C) \to H^1(G',\C)$ is an isomorphism.  Then Lemma \ref{lem cover} implies that this linear isomorphism  identifies $ \cR^1(G,\C) $ with $\cR^1(G',\C),$ hence $(\pi^*)^{-1}(\TC_1 V')\subseteq \cR^1(G,\C) $. Then by Lemma \ref{lem from char to H^1} we get that $$L_{u,v}\otimes_\Q \C  \subsetneq (\pi^*)^{-1}(\TC_1 V')\subseteq \cR^1(G,\C) ,$$ since $ \widehat{p'} (T_\lambda) \subsetneq V.$  But  Lemma \ref{lem:multipartite} and Lemma \ref{lem:no_cross_weight} imply that all $x_i$'s are adjacent to both $u$ and $v$. Then
Corollary \ref{cor  char} shows that $L_{u,v}\otimes_\Q \C \nsubseteq \mathcal{R}^1(X,\C)$, which gives a contradiction.

Next by \Cref{thm:structure}, $\widehat{p'} (T_\lambda)$ being an irreducible component of $\cV^1(G',\C)$ implies that 
 there exists an algebraic map $f\colon X' \to C$  such that 
$\widehat{p'} (T_\lambda)$is one of connected components of $f^* \cV^1(\pi_1^\orb(C,\bar m), \C)$. 
Again $C$ is compact for the same reason as in Step A. Hence $C$ is an elliptic curve. 
Note that by taking a torsion translate, we have $\widehat{p'} \Char^0(Q',\C)=f^*(\Char^0(\pi_1^\orb(C,\bar m), \C))$. Then by Lemma \ref{lem from char to group}, 
 the image of the induced injective map $f^*\colon H^1(C,\Q) \to H^1(X',\Q)  $ coincides with the image of $(p')^* \colon H^1(Q',\Q) \to H^1(G',\Q)$.   
Hence the claim follows, since $f$ is algebraic and $H^1(C,\Q)$ has pure Hodge structure of weight one. 
\end{proof}

Now we are ready to finish step 3.
\begin{lemma}\label{lem:2edge}
If two heavy edges \(\{u,v\}\) and \(\{u,w\}\) share the same vertex \(u\), then
\(G_{\Gamma^\ell}\) is not quasi-projective.
\end{lemma}

\begin{proof}
Suppose  that $G_{\Gamma^\ell}=\pi_1(X)$
 for some smooth quasi-projective variety \(X\). By Lemma \ref{lem Hodge}, both $L_{u,v}$ and $L_{u,w}$ have a pure Hodge structure of weight one. Then
their intersection 
\[
L_{u,v}\cap L_{u,w} = \mathbb Q\langle u^*\rangle
\]
is a one-dimensional rational pure Hodge structure of weight one, which is impossible.
\end{proof}

\subsection{The classification and its consequences}
\begin{proof}[Proof of Theorem \ref{thm:main}]
    It is clear that (ii) $\Leftrightarrow$ (iii).
    
    (i) $\Rightarrow$ (iii): 
    Lemma \ref{lem:multipartite} makes the underlying graph complete multipartite. Lemma \ref{lem:no_cross_weight} shows that $\Gamma_\ell$ is a join of some discrete graphs with at least two vertices and a complete graph $K$ (possibly with weighted edges). Finally, Lemma \ref{lem:2edge} proves that within $K$,  two heavy edges can not share a vertex. Thus, $K$ is  a finite join of either $S_m$ or a graph consisting of a single vertex. 

    (ii) $\Rightarrow$ (i): It is clear that the free group $\F_r$ (with $r\geq 1$) is the fundamental group of the variety $\C \setminus \{r \text{ points}\}$. On the other hand, the group  $ \langle a, b \mid [a,b]^m = 1 \rangle$ is indeed a projective group, see e.g. \cite[Section~5.3]{BM}. Since the graph is a join of $D_r$'s and $S_m$'s, the overall group is the direct product of their respective groups. The product of these varieties is smooth
quasi-projective, with the required fundamental group, see Proposition \ref{prop property}(2).
\end{proof}

As an application, we classify the weighted right-angled Artin groups realized by some smooth quasi-projective variety $X$ having $H^1(X,\Q)$ of pure Hodge structure of weight two, such as the complement of plane curves in $\mathbb{P}^2$ or a hyperplane arrangement complement.

\begin{corollary} \label{cor arrangement}
        Let $G=G_{\Gamma_\ell}$ be a weighted right-angled Artin group. The following are equivalent:
       \begin{enumerate}[label=(\roman*)]
        \item $G$ is the fundamental group of a complex hyperplane arrangement complement;
        \item  $G\cong \pi_1(X)$ for a complex smooth quasi-projective variety $X$ with $W_1(H^1(X,\Q)) = 0$;
        \item $G$ is a finite product of finitely generated free groups;
        \item $\Gamma_\ell$ is complete multipartite and all edge weights are  one.
    \end{enumerate}
\end{corollary}
\begin{proof}
  It is clear that   $(i)\Rightarrow (ii)$ and $ (iii) \iff (iv) $. 
  
    For  $(iii)\Rightarrow$ (i), the finitely generated free group $\F_r$ is the fundamental group of $\C \setminus \{r \text{ points}\}$, hence the fundamental group of hyperplane arrangement complement. A finite product of hyperplane arrangement complements is still a hyperplane arrangement complement. Hence the claim follows.

It remains to prove $(ii) \Rightarrow (iii)$. By Lemma \ref{lem:multipartite}, the underlying graph \(\Gamma\) is complete multipartite. It suffices to
show that all edge weights are equal to \(1\).
Suppose that \(\{u,v\}\) is a heavy edge. By Lemma \ref{lem Hodge}, $L_{u,v}\subseteq H^1(X,\Q)$ has a pure Hodge structure of weight one. This is impossible since we assumed that $W_1 H^1(X,\Q)=0$.  
\end{proof}

\begin{corollary} \label{cor projective}
        Let $G=G_{\Gamma_\ell}$ be a weighted right-angled Artin group. The following are equivalent.
      \begin{enumerate}[label=(\roman*)]
        \item  $G$ is a projective group;
        \item $G$ is the product of groups of type  $Q_m$ for some $m\geq 1$
        \item $\Gamma_\ell$ is a finite join of graphs of type $S_{m}$.
    \end{enumerate}
\end{corollary}
\begin{proof}
  It is clear that    $ (ii) \iff (iii) $. Note that $ Q_m$ is a projective group, see e. g. \cite[Section~5.3]{BM}. Hence $(ii)\Rightarrow (i)$. 

It remains to prove $(i) \Rightarrow (ii)$. Note that $G$ is a projective group, hence quasi-projective. By \Cref{thm:main}, $G$ is the product of groups of type $\F_r$  or $ Q_m$. If one of the product factor is $\F_r$ with $r>1$, which corresponds that $\Gamma_\ell$ has a maximal discrete subgraph $\Gamma(W)$ with $r$ vertices, then  $T_W$ is a irreducible component of $\cV^1(G,\C)$. In particular, $L_W$ is 0-isotropic. But this contradicts Remark \ref{rem isotropic}. Hence  $G_{\Gamma^{\ell}}
\cong
\Z^n
\times
\prod_{\beta}
\left\langle
a_{\beta},b_{\beta}
\ \middle|\
[a_{\beta},b_{\beta}]^{m_{\beta}}=1
\right\rangle$. 
  Since $G$ is projective, $n$ is even. Note that $\Z^2=\langle a,b | [a,b]=1 \rangle=Q_1$. The claim follows.   
\end{proof}

\begin{remark}\label{rem:kahler}
The same method gives an alternative proof of the K\"ahler classification for weighted right-angled Artin groups in \cite[Theorem 1.6]{LL}. 
Both Theorem \ref{thm:structure} and Theorem \ref{thm:DPS-isotropic} apply to compact K\"ahler manifold, see e. g. \cite[Theorem 9.3]{Suc14} and \cite[Theorem C]{DPS}. 
Then Lemma \ref{lem:multipartite} and Lemma \ref{lem Hodge} still apply, and Step 2 is unnecessary. Hence the classification follows by the same proof as for Corollary \ref{cor projective}. Thus, within this class, the K\"ahler and projective groups coincide.
\end{remark}

\end{document}